\documentclass[reqno]{amsart}

\usepackage{amsmath,amssymb,color}
\usepackage{amsthm}
\usepackage{enumerate,enumitem}
\usepackage{graphicx,subfigure}
\usepackage{verbatim,mathrsfs}
\usepackage{cite}

\DeclareGraphicsExtensions{.pdf}
\graphicspath{{figs/}}

\newtheorem{theorem}{Theorem}
\newtheorem{lemma}[theorem]{Lemma}
\newtheorem{proposition}[theorem]{Proposition}
\newtheorem{corollary}[theorem]{Corollary}
\newtheorem{problem}[theorem]{Problem}
\newtheorem*{claim}{Claim}

\newcommand{\cS}{\mathcal{S}}

\newcommand{\cH}{\mathcal{H}}
\newcommand{\cR}{\mathcal{R}}

\newcommand{\cA}{\mathcal{A}}
\newcommand{\cE}{\mathcal{E}}

\newcommand{\Real}{\ensuremath{\mathbb{R}}}
\newcommand{\n}{\ensuremath{\operatorname{next}}}

\newenvironment{claimproof}{%
\noindent%
\textit{Proof of Claim.}%
}
{
\hfill $\triangle$%
\medskip
}

\begin{document}

\title{Density of Range Capturing Hypergraphs}

\author[M.~Axenovich]{Maria Axenovich}
\author[T.~Ueckerdt]{Torsten Ueckerdt}

\address[M.~Axenovich]{Karlsruher Institut f\"ur Technologie, Karlsruhe, Germany}
\email{maria.aksenovich@kit.edu}

\address[T.~Ueckerdt]{Karlsruher Institut f\"ur Technologie, Karlsruhe, Germany}
\email{torsten.ueckerdt@kit.edu}


\begin{abstract}
 For a finite set $X$ of points in the plane, a set $S$ in the plane, and a positive integer $k$, we say that a $k$-element subset $Y$ of $X$ is captured by $S$ if there is a homothetic copy $S'$ of $S$ such that $X\cap S' = Y$, i.e., $S'$ contains exactly $k$ elements from $X$.
 A $k$-uniform $S$-capturing hypergraph $\cH = \cH(X,S,k)$ has a vertex set $X$ and a hyperedge set consisting of all $k$-element subsets of $X$ captured by $S$.
 In case when $k=2$ and $S$ is convex these graphs are planar graphs, known as \emph{convex distance function Delaunay graphs}.
 
 In this paper we prove that for any $k\geq 2$, any $X$, and any convex compact set $S$, the number of hyperedges in $\cH(X,S,k)$ is at most $(2k-1)|X| - k^2 + 1 - \sum_{i=1}^{k-1}a_i$, where $a_i$ is the number of $i$-element subsets of $X$ that can be separated from the rest of $X$ with a straight line.
 In particular, this bound is independent of $S$ and indeed the bound is tight for all ``round'' sets $S$ and point sets $X$ in general position with respect to $S$.
 
 This refines a general result of Buzaglo, Pinchasi and Rote~\cite{BuPiRo13} stating that every pseudodisc topological hypergraph with vertex set $X$ has $O(k^2|X|)$ hyperedges of size $k$ or less.
\end{abstract}

\maketitle


\section{Introduction}
 
Let $S$ and $X$ be two subsets of the Euclidean plane $\Real^2$ and $k$ be a positive integer.
In this paper, $S$ is a convex compact set containing the origin and $X$ is a finite set.
An {\bf $S$-\emph{range}} is a homothetic copy of $S$, i.e., a set obtained from $S$ by scaling with a positive factor with respect to the origin and an arbitrary translation.
In other words, an $S$-range is obtained from $S$ by first contraction or dilation and then translation, where two $S$-ranges are \emph{contractions/dilations} of one-another if in their corresponding mappings the origin is mapped to the same point.

We say that an $S$-range $S'$ \emph{captures} a subset $Y$ of $X$ if $X\cap S' = Y$.
An {\bf $S$-capturing hypergraph} is a hypergraph $\cH = (X,\cE)$ with vertex set $X$ and edge set $\cE \subseteq 2^X$ such that for every $Y\in \cE$ there is an $S$-range $S'$ that captures $E$.

In this paper we consider {\bf $k$-uniform} $S$-capturing hypergraphs, that is, those hypergraphs $\cH = \cH(X,S,k)$ with vertex set $X$ and hyperedge set consisting of \emph{all} $k$-element subsets of $X$ captured by $S$.
I.e., the hyperedges correspond to $S$-ranges containing exactly $k$ elements from $X$.
These hypergraphs are often referred to as \emph{range-capturing hypergraphs} or \emph{range spaces}.
The importance of studying $k$-uniform $S$-capturing hypergraphs was emphasized by their connection to epsilon nets and covering problems of the plane~\cite{PT11,Pa86,Mat}.
See also some related literature for geometric hypergraphs,~\cite{KNPS06,KKN04,BuPiRo13,wood2013hypergraph,PaSh09,BMP,KP12,FP84,PT10,P13,PT10b,CPST09,TT07,CKMU12,PW90, Smo}.

The first non-trivial case $k=2$, i.e., when $\cH(X,S,k)$ is an ordinary graph, was first considered by Chew and Dyrsdale in 1985~\cite{CD85}.
They showed that if $S$ is convex and compact, then $\cH(X,S,2)$ is a planar graph, called the \emph{Delaunay graph} of $X$ for the \emph{convex distance function} defined by $S$.
In particular, $\cH(X,S,2)$ has at most $3|X|-6$ edges and this bound can be achieved.
We remark that it follows from Schnyder's realizer~\cite{S90} that every maximally planar graph can be written as $\cH(X,S,2)$ for some $X$ and $S$ being any triangle.

\subsection{Related work.}
Recently, Buzaglo, Pinchasi and Rote~\cite{BuPiRo13} considered the maximum number of hyperedges of size $k$ or less in a pseudodisc topological hypergraph on $n$ vertices.
Here, a \emph{family of pseudodiscs} is a set of closed Jordan curves such that any two of these curves either do not intersect or intersect in exactly two points. 
A hypergraph is called \emph{pseudodisc topological hypergraph} if its vertex set $X$ is a set of points in the plane and for every hyperedge $Y$ there is a closed Jordan curve such that the bounded region of the plane obtained by deleting the curve contains $Y$ and no point from $X \setminus Y$, and the set of all these Jordan curves is a family of pseudodiscs.

The authors of~\cite{BuPiRo13} observed that pseudodisc topological hypergraphs have VC-di\-men\-sion~\cite{VC} at most~$3$, and that using this fact the number of hyperedges can be bounded from above.
For this, a version of the Perles-Sauer-Shelah theorem~\cite{Sa,Shelah} is applied.
Let, for a set $A$ and a positive integer $d$, $\binom{A}{\leq d}$ denote the set of all subsets of $A$ of size at most $d$.

\begin{theorem}[Perles-Sauer-Shelah Theorem]
 Let $F = \{A_1, \ldots ,A_m\}$  be a family of distinct subsets of $\{1,2,\ldots,n\}$ and let $F$ have VC-dimension at most $d$. Then $m\leq \left|\bigcup_{i=1}^{m}\binom{A_i}{\leq d}\right|.$  
\end{theorem}

Applying this theorem to the family of hyperedges in a pseudodisc topological hypergraph, one can see that the number of hyperedges in such a hypergraph is at most $O(n^3)$.
In fact, if one considers only hyperedges of size $k$ or less, a much stronger bound could be obtained. 
 
\begin{theorem}[Buzaglo, Pinchasi and Rote~\cite{BuPiRo13}]\label{thm:topological-density}
 Every pseudodisc topological hypergraph on $n$ vertices has $O(k^2n)$ hyperedges of size $k$ or less.
\end{theorem}

However, the methods used to prove Theorem~\ref{thm:topological-density} do not seem to give any non-trivial bound on the number of hyperedges of size \emph{exactly $k$}.
Tight bounds are only known is case $k=2$.
Indeed, every $2$-uniform $S$-capturing hypergraph is a planar graph~\cite{CD85}, called the \emph{convex distance function Delaunay graph}, and thus has at most $3n-6$ edges.

\subsection{Our results.}
In this paper, we consider the case when every hyperedge has {\bf exactly $k$ points}.
In particular, we consider $k$-uniform $S$-capturing hypergraphs, for convex and compact sets $S$.
One can show that these hypergraphs are pseudodisc topological hypergraphs.
Indeed, the family of all homothetic copies of a fixed convex set $S$ is surely the most important example of a family of pseudodiscs.
For a finite point set $X$ of the plane, a subset $Y$ of $X$ can be \emph{separated} with a straight line if there exists a line $\ell$ such that one halfplane defined by $\ell$ contains all points in $Y$ and the other contains all points in $X - Y$.
For a positive integer $i$, let $a_i$ denote the number of $i$-subsets of $X$ that can be separated with a straight line.

\begin{theorem}\label{thm:main}
 Let $S$ be a convex compact set, $X$ be a finite point set and $k$ be a positive integer.
 Any $k$-uniform $S$-capturing hypergraph on vertex set $X$ has at most $(2k-1)|X| - k^2 + 1 - \sum_{i=1}^{k-1}a_i$ hyperedges.
 Moreover, equality holds whenever $S$ is nice and $X$ is in general position with respect to $S$.
\end{theorem}

Here a set $S$ is called \emph{nice} if $S$ has ``no corners'' and ``no straight segments on its boundary''.
We define such nice shapes formally later.
Moreover, $X$ is in \emph{general position with respect to $S$} if no three points of $X$ are collinear and no four points of $X$ lie on the boundary of any $S$-range.

Note that for $k=2$ the bound in Theorem~\ref{thm:main} amounts for at most $3|X|-3-t$ edges, where $t = a_1$ is the number of corners of the convex hull of $X$.
We obtain the following refinement of Theorem~\ref{thm:topological-density}.

\begin{corollary}
 Let $S$ be a convex compact set and $k,n$ be positive integers.
 Any $S$-capturing hypergraph on $n$ vertices has at most $k^2n + O(k^3)$ hyperedges of size $k$ or less.
\end{corollary}

The paper is organized as follows.
Section~\ref{sec:notation} provides general definitions.
Here we also show how to reduce the general case of an arbitrary capturing hypergraph to one with a nice shape $S$ and a point set $X$ in general position with respect to $S$.
Section~\ref{sec:types} introduces different types of ranges.
The number of ranges of Type~I is determined exactly in Section~\ref{sec:type-I}.
Section~\ref{sec:exact-bound} gives an identity involving the number of ranges of both types in $X$.
Finally, Theorem~\ref{thm:main} is proven in Section~\ref{sec:upper-bound}.

\section{Nice shapes, general position and next range}\label{sec:notation}

In this section we introduce nice shapes, the concepts of the next range and state their basic properties.
For the ease of reading, the proofs of some results in this section are provided in the appendix because they are quite straightforward but also technical.
We denote the boundary of a set $S$ by $\partial S$.
We denote the line through distinct points $p$ and $q$ by $\overline{pq}$.
A halfplane defined by a line $\ell$ is a connected component of the plane after the removal of $\ell$.
In particular, such halfplanes are open sets.
A closed halfplane is the closure of a halfplane, i.e., the union of the halfplane and its defining line.
Typically we denote the two halfplanes defined by a line by $L$ and $R$ (standing for ``left'' and ``right'').
 
For a set $X$ of $n$ points in the plane and any $i \in [n]$ an \emph{$i$-set} of $X$ is a subset $Y$ of $X$ on $i$ elements that can be separated with a straight line.
In other words, $Y$ is an $i$-set if it is captured by a closed halfplane.
The number of $i$-sets of $X$ is denoted by $a_i$.
Note that some (but in general not all) $i$-sets can be captured by a halfplane that has two points of $X$ on its boundary.
Such halfplanes are called \emph{representative halfplanes} and we denote the set of all representative halfplanes of $i$-sets of $X$ by $\cA_i$.
Note that (even if $i \geq 2$) the number of representative halfplanes for a fixed $i$-set might be anything, including $0$.
 
\begin{lemma}\label{lem:halfplanes}
 For any set $X$ of $n$ points in the plane, no three on a line, and any $i \in \{1,\ldots,n-1\}$ we have $a_i = |\cA_{i+1}|$.
\end{lemma}
\begin{proof}
 Let $X$ be a finite point set with no three points in $X$ on a line, $|X| = n$ and $i \in \{1,\ldots,n-1\}$.
 We shall give a bijection between the set $X_i$ of $i$-sets of $X$ and $\cA_{i+1}$.
 
 Let $Y \in X_i$ be an $i$-set.
 Assume (by rotating the plane if needed) that $Y$ is separated from $X - Y$ by a vertical line, such that $Y$ is contained in the corresponding right halfplane.
 Consider all closed halfplanes that contain all points in $Y$ and whose interior does not contain any point in $X - Y$.
 Among all lines defining such halfplane let $\ell$ be one with smallest slope.
 Then $\ell \cap X$ contains a point $p \in X - Y$ and a point $q \in Y$ and going from $p$ to $q$ along $\ell$ we have the corresponding halfplane that contains $X$ on the right.
 In particular this right halfplane of $\ell$ is in $\cA_{i+1}$.
 
 On the other hand, for any closed halfplane $H \in \cA_{i+1}$ consider the line $\ell$ defining $H$ and the two points $p,q \in X \cap \ell$ so that going from $p$ to $q$ along $\ell$ we have $H$ on the right.
 Then rotating $\ell$ slightly counterclockwise around any point on $\ell$ between $p$ and $q$ shows that $(H \cap X) - p$ is an $i$-set of $X$.
 
 The above bijection shows that $a_i = |X_i| = |\cA_{i+1}|$, as desired.
\end{proof}

\subsection{Nice shapes and general position of a point set.}
A convex compact set $S$ is called a {\bf nice shape} if
\begin{enumerate}[label = (\roman*)]
 \item for each point in $\partial S$ there is exactly one line that intersects $S$ only in this point and
 \item the boundary of $S$ contains no non-trivial straight line segment.
\end{enumerate}

For example, a disc is a nice shape, but a rectangle is not.
A nice shape has no ``corners'' and we depict nice shapes as discs in most of the illustrations.

\begin{lemma}\label{lem:small-facts}
 If $S$ is a nice shape, $S_1$ and $S_2$ are distinct $S$-ranges, then each of the following holds.
 \begin{enumerate}[label = (\roman*)]
  \item  $\partial S_1 \cap \partial S_2$ is a set of at most two points.\label{enum:pseudodisc}
  
  \item If $\partial S_1 \cap \partial S_2 = \{p,q\}$ and $L$ and $R$ are the two open halfplanes defined by $\overline{pq}$, then \label{enum:subset-halfplane}
   \begin{itemize}
    \item $S_1 \cap L \subset S_2 \cap L$ and $S_1 \cap R \supset S_2 \cap R$ or
    \item $S_1 \cap L \supset S_2 \cap L$ and $S_1 \cap R \subset S_2 \cap R$.
   \end{itemize}
  
  \item Any three non-collinear points lie on the boundary of a unique $S$-range.\label{enum:unique-range}
    
  \item For a subset of points $X \subset \Real^2$ and any $Y \subset X$, $|Y| \geq 2$, that is captured by some $S$-range there exists at least one $S$-range $S'$ with $Y = X \cap S'$ and $|\partial S' \cap X| \geq 2$.\label{enum:two-on-boundary}
 \end{enumerate}
\end{lemma}

The proof of Lemma~\ref{lem:small-facts} is provided in the Appendix.
We remark that only the last item of Lemma~\ref{lem:small-facts} remains true if $S$ is convex compact but not nice.
For example, if $S$ is an axis-aligned square, then no three points with strictly monotone $x$- and $y$-coordinates lie on the boundary of any $S$-range, whereas three points, two of which have the same $x$- or $y$-coordinate lie on the boundary of infinitely many $S$-ranges.

\medskip

\noindent
For sets $X,S \subset \Real^2$ we say that {\bf $X$ is in general position with respect to $S$} if 
\begin{enumerate}[label = (\roman*)]
 \item no two points of $X$ are on a vertical line,
 \item no three points of $X$ are collinear,
 \item no four points of $X$ lie on the boundary of any $S$-range. 
\end{enumerate}

\begin{lemma}\label{lem:general-position} 
 For any point set $X$, positive integer $k$ and a convex compact set $S$, there is a nice shape $S'$ and a point set $X'$ in general position with respect to $S'$, such that $|X'|=|X|$ and the number of edges in $\cH(X',S',k)$ is at least as large as the number of edges in $\cH(X,S,k)$.
\end{lemma}

To prove Lemma~\ref{lem:general-position}, we show that one can move the points of $X$ slightly and modify $S$ slightly to obtain the desired property.
See the Appendix for a detailed account of the argument.
From now on we will always assume that $S$ is a nice shape and $X$ is a finite point set in general position with respect to $S$.

\subsection{Next Range.}
For two distinct points $p,q$ in the plane, we define $\cS_{pq}$ to be the set of all $S$-ranges $S_1$ with $p,q \in \partial S_1$.
The symmetric difference of two sets $A$ and $B$ is given by $A \triangle B = (A \setminus B) \cup (B \setminus A)$.

\begin{lemma}\label{lem:next-S-range}
 Let $p,q$ be two points such that no four points in $X \cup \{p,q\}$ lie on the boundary of an $S$-range.
 Let $L$ be a halfplane defined by $\overline{pq}$.
 Then the following holds.
 \begin{enumerate}[label = (\roman*)]
  \item The $S$-ranges in $\cS_{pq}$ are linearly ordered, denoted by $\prec_{pq}$, by inclusion of their intersection with $L$:
  \[
   S_1 \prec_{pq} S_2 \quad \Leftrightarrow \quad S_1 \cap L \subset S_2 \cap L \qquad \text{for all } S_1,S_2 \in \cS_{pq}
  \]
  \item For each $S_1 \in \cS_{pq}$ there exists a $\prec_{pq}$-minimal $S_2 \in \cS_{pq}$ with
  \[
   (\partial S_2 \setminus \partial S_1) \cap X \neq \emptyset \qquad \text{and} \qquad S_1 \prec_{pq} S_2
  \]
  if and only if $S_1 \triangle L$ contains a point from $X$ in its interior.\label{enum:next-exists}
 \end{enumerate}
\end{lemma}

The proof of Lemma~\ref{lem:next-S-range} is provided in the Appendix.

\medskip

\noindent
Whenever no four points in $X \cup \{p,q\}$ lie on the boundary of an $S$-range, $L$ is a halfplane defined by $\overline{pq}$ and $S_1 \in \cS_{pq}$, we define $\n_L(S_1)$, called the {\bf next range} of $S_1$ in $L$, as follows.
\begin{itemize}
 \item If the interior of $S_1 \triangle L$ contains a point from $X$, then $\n_L(S_1) = S_2$ for the $S$-range $S_2 \in \cS_{pq}$ in Lemma~\ref{lem:next-S-range}~\ref{enum:next-exists}.
 \item If the interior of $S_1 \triangle L$ contains no point from $X$, then $\n_L(S_1) = L$.
\end{itemize}
Informally, we can imagine continuously transforming $S_1$ into a new $S$-range containing $p$ and $q$ on its boundary and containing $S_1\cap L$ and all points of $(S_1 \setminus \partial S_1) \cap X$.
As soon as this new $S$-range contains a point from $X \setminus \{p,q\}$ on its boundary, we choose it as the next range of $S_1$ in $L$.
Note that if $S_2 = \n_{L}(S_1)$ and $\partial S_1$ contains a point from $X \setminus \{p,q\}$, then $\n_{R}(S_2) = S_1$, where $R$ denotes the other halfplane defined by $\overline{pq}$.

As no four points in $X \cup \{p,q\}$ lie on the boundary of an $S$-range, we have $|\partial S_1 \cap (X \setminus \{p,q\})| \leq 1$ for each $S_1 \in \cS_{pq}$.
This implies that if $S_1$ captures $k$ elements of $X$ then $\n_{L}(S_1)$ captures $k-1$, $k$ or $k+1$ elements of $X$.
Indeed, if $\n_L(S_1) = L$, then $|X \cap L| = k$ and if $\n_L(S_1) \neq L$, then the following holds.
\begin{align}
 \text{If } R \cap \partial S_1 \cap X \neq \emptyset, \text{ then } \n_{L}(S_1) \text{ captures $k$ or $k-1$ points.} \label{eq:capture-minus-1}\\
 \text{If } R \cap \partial S_1 \cap X = \emptyset, \text{ then } \n_{L}(S_1) \text{ captures $k$ or $k+1$ points.} \label{eq:capture-plus-1}
\end{align}
See Figure~\ref{fig:next-range} for the three possible case scenarios.

\begin{figure}[htb]
 \centering
 \includegraphics{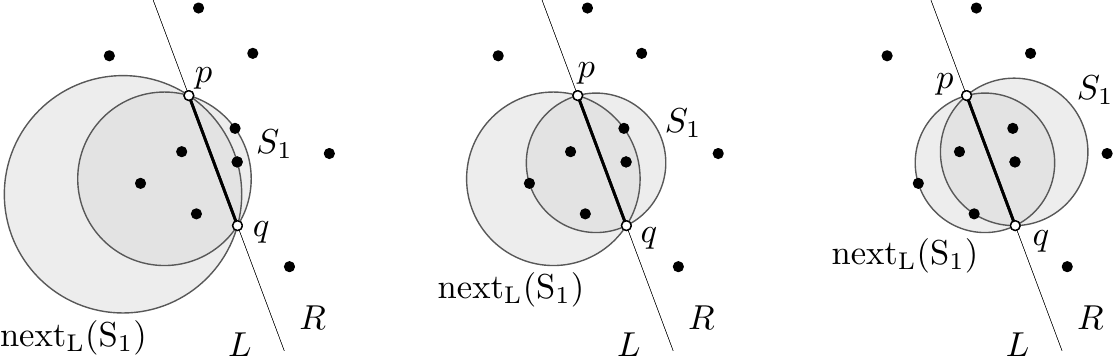}
 \caption{Three cases of an $S$-range $S_1$ with two boundary points $p$ and $q$, and the next $S$-range of $S_1$ in a halfplane $L$ defined by $\overline{pq}$. Note that $|X \cap \n_{L}(S_1)| - |X \cap S_1|$ is $-1$ on the left, $0$ in the middle, and $1$ on the right.}
 \label{fig:next-range}
\end{figure}

\section{Representative $S$-ranges and Types I and II}\label{sec:types}
 
Let $X$ be a set in a general position and $S$ be a nice shape. Let $Y$ be any hyperedge in $\cH(X,S,k)$.
An $S$-range $S'$ is a \emph{representative $S$-range for $Y$} if $Y = X \cap S'$ and among all such $S$-ranges $S'$ has the maximum number of points from $Y$ on its boundary.
From Lemma~\ref{lem:small-facts}~\ref{enum:two-on-boundary} it follows that each hyperedge has a representative range and if $S'$ is a representative $S$-range for $Y$, then $S'$ has two or three points of $X$ on its boundary.
 
\medskip
 
\noindent
We say that $S'$ is of {\bf Type~I} if $|\partial S' \cap X| = 3$ and of {\bf Type~II} if $|\partial S' \cap X| = 2$.

\medskip

\noindent
We say that $Y$ is of \emph{Type~I} if it has a representative range of Type~I, otherwise it is of \emph{Type~II}.
Note that in total we have at most $\binom{k}{3}$ many Type~I ranges representing a Type~I hyperedge since by Lemma~\ref{lem:small-facts}~\ref{enum:unique-range} any three points of $X$ are on the boundary of only one $S$-range.
On the other hand, every Type~II hyperedge has infinitely many representative ranges, see Figure~\ref{fig:representative-ranges}.
The {\bf representative set} of $Y$ contains all representative ranges for a Type~I set $Y$ and it contains one arbitrarily chosen representative range for a Type~II set $Y$.
We denote a representative set of $Y$ by $\cR(Y)$.

\begin{figure}[htb]
 \centering
 \subfigure[\label{fig:representative-type1}]{
  \includegraphics{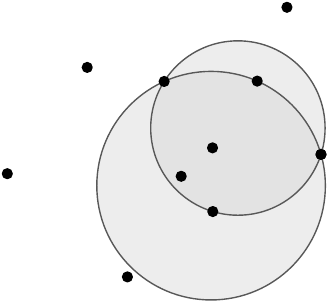}
 }
 \hspace{3em}
 \subfigure[\label{fig:representative-type2}]{
  \includegraphics{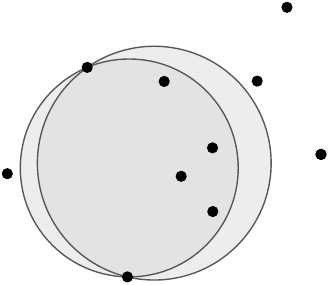}
 }
 \caption{\subref{fig:representative-type1} A Type~I hyperedge and with only two representative ranges. \subref{fig:representative-type2} A Type~II hyperedge with two of its infinitely many representative ranges.}
 \label{fig:representative-ranges}
\end{figure}

We define
\begin{align*}
 \cE^k_1 &= \{ Y \subseteq X \mid Y \mbox{ is a Type~I hyperedge}\},\\
 \cR^k_1 &= \{ S\in \cR(Y) \mid Y \in \cE^k_1 \},\\
 \cE^k_2 &= \{ Y \subseteq X \mid Y \mbox{ is a Type~II hyperedge}\} \mbox{ and}\\
 \cR^k_2 &= \{ S\in \cR(Y) \mid Y \in \cE^k_2\}.
\end{align*}

\begin{lemma}\label{lem:same-boundary-pair}
 All Type~II representative $S$-ranges for the same hyperedge have the same pair of $X$ in their boundary.
\end{lemma}
\begin{proof}
 Consider a representative range $S_1$ for a Type~II hyperedge $Y$ with $\{p_1, q_1\} =Y\cap \partial S_1$.
 Assume for the sake of contraction that there is another representative range, $S_2$ for $Y$ with $\{p_2, q_2\}= Y \cap S_2$ and $\{p_2,q_2\} \neq \{p_1,q_1\}$.
 We have that $p_2,q_2 \in S_1$, $p_1,q_1\in S_2$.
 Assume, without loss of generality that $q_2\not\in \{p_1, q_1\}$ and $q_1 \not\in \{p_2, q_2\}$.
 Then $q_2 \in S_1 - \partial S_1$ and $q_1 \in S_2 - \partial S_2$.
  
 \begin{figure}[htb]
  \centering
  \includegraphics{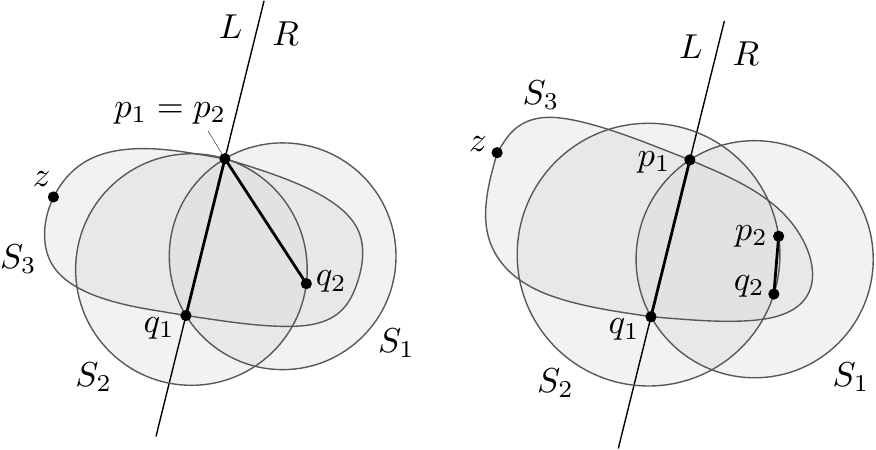}
  \caption{The three cases in the proof of Lemma~\ref{lem:same-boundary-pair}.}
  \label{fig:same-boundary-pair}
 \end{figure}
  
 We need to distinguish the following cases: segments $p_2q_2$ and $p_1q_1$ cross properly, $p_2q_2$ and $p_1q_1$ share a vertex, i.e., $p_2=p_1$, and finally $p_1q_1$ is to the left of $p_2q_2$.
 The first case does not occur by the argument presented in the previous item.
 In the other two cases, let $L$ and $R$ be the halfplanes defined by $\overline{p_1q_1}$ \emph{not} containing $q_2$ and containing $q_2$, respectively.
 Consider the $S$-range $S_3 = \n_L(S_1)$, which exists by Lemma~\ref{lem:next-S-range} as $S_1 \Delta L$ contains $q_2$.
 We see that $R$ contains at least one point in $\partial S_3 \cap \partial S_2$, see Figure~\ref{fig:same-boundary-pair}.
 Moreover, $S_3$ must contain a point $z \in X \setminus Y$ because it is a next range and $Y$ is not of Type~I.
 It follows that $z \in L \setminus S_2$, which implies that the closure of $L$ contains two points in $\partial S_3 \cap \partial S_2$, a contradiction to Lemma~\ref{lem:small-facts}~\ref{enum:pseudodisc}.
\end{proof}

From Lemma~\ref{lem:same-boundary-pair} and the definitions above we immediately get the following.

\begin{align}
  |\cE^k_1| &= |\cR^k_1| - \sum_{Y \in \cE^k_1}(|\cR(Y)| - 1)\label{eq:clearly-1} \\
  |\cE^k_2| &= |\cR^k_2|\label{eq:clearly-2}
 \end{align} 

For a Type~I hyperedge $Y$, let $Y' \subseteq Y$ be the subset of vertices that are on the boundary of at least one representative range for $Y$.
We define the graph $G(Y) = (Y',E_Y)$ with vertex set $Y'$ and edge set $E_Y=\{\{p,q\} \mid p,q \in \partial S', ~S' \in \cR(Y)\}$.
Then $G(Y)$ is the union of triangles, one for each representative range of $Y$.
We call an edge of $G(Y)$ \emph{inner edge} if it is contained in at least two triangles.

\begin{lemma}\label{lem:outerplanar-graph}
 For every Type~I hyperedge $Y$ the graph $G(Y)$ is a maximally outerplanar graph.
 In particular, $G(Y)$ has exactly $|\cR(Y)|-1$ inner edges.
\end{lemma}
\begin{proof}
 We shall show that $G = G(Y)$ is maximally outerplanar by finding a planar embedding of $G$ in which every vertex lies on the outer face, every inner edge lies in two triangles and every outer edge lies in one triangle.
 To this end draw every vertex of $G$ at the position of its corresponding point in $Y'$ and every edge as a straight-line segment.
 
 First observe that every point in $Y'$ lies on the convex hull of $Y'$.
 Hence, if $G$ is drawn crossing-free, then every vertex of $G$ lies on the outer face of $G$.

 Second, assume for the sake of contraction that two edges $x_1y_1, x_2y_2 \in E(G)$ cross.
 Without loss of generality these four vertices appear on the convex hull of $Y'$ in the clockwise order $x_1,x_2,y_1,y_2$.
 But then the $S$-ranges $S_1$ and $S_2$ with $Y \subset S_i$ and $\{x_i,y_i\} = \partial S_i \cap Y$ for $i=1,2$ have at least four intersections on their boundaries.
 See Figure~\ref{fig:outerplanar-2} for an illustration.
 This is a contradiction to Lemma~\ref{lem:small-facts}~\ref{enum:pseudodisc}, i.e., that $|\partial S_1 \cap \partial S_2| \leq 2$.
  
 \begin{figure}[htb]
  \centering
  \subfigure[\label{fig:outerplanar-2}]{
   \includegraphics{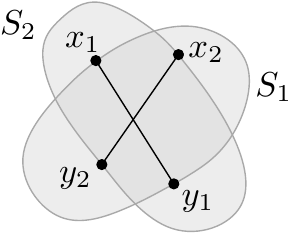}
  }
  \hspace{2em}
  \subfigure[\label{fig:outerplanar-3}]{
   \includegraphics{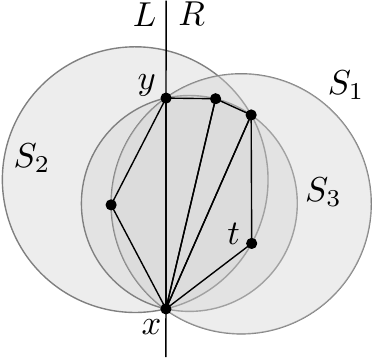}
  }
  \caption{
  \subref{fig:outerplanar-2} If two edges $x_1y_1$ and $x_2y_2$ in $G(Y)$ cross, then the corresponding $S$-ranges have at least four intersections on their boundaries.
  \subref{fig:outerplanar-3} Illustration of the proof of Lemma~\ref{lem:outerplanar-graph}.}
 \end{figure}
 
 Finally, for any edge $xy$ in $G$ we shall show that $x$ and $y$ are consecutive points on the convex hull of $Y'$, or $xy$ is contained in two triangles.
 Note that this proves that $G$ is maximally outerplanar.
 
 As $xy$ is an edge in $G$, there is a representative range $S_1 \in \cR(Y)$ with $x,y \in \partial S_1$.
 Let $L,R$ be the two open halfplanes defined by $\overline{xy}$.
 Assume without loss of generality that the third point in $\partial S_1 \cap X$ lies in $L$.
 We distinguish two cases.
 
 If $R \cap Y' = \emptyset$, then $x$ and $y$ are consecutive on the convex hull of $Y'$ and the edge $xy$ is an outer edge of $G$.
 Otherwise, there exists some point $t \in R \cap Y'$ and we shall show that the edge $xy$ lies in two triangles.
 Let $S_2 \in \cR(Y)$ be a representative range with $t \in \partial S_2$, which exists as $t \in Y'$.
 Also consider $S_3 = \n_L(S_1)$.
 By~\eqref{eq:capture-plus-1} we have $Y \subset S_3$ and $|X \cap S_3| \in \{k,k+1\}$.
 If $S_2 = S_3$, then the edge $xy$ lies in two triangles, one for $S_1$ and one for $S_2 = S_3$.
 Otherwise the situation is illustrated in Figure~\ref{fig:outerplanar-3}.
 We have $x,y \in S_2$ and $t \in S_3 - \partial S_3$.
 It follows that $\partial S_2$ and $\partial S_3$ intersect (at least) twice in the closure of $R$.
 By Lemma~\ref{lem:small-facts}~\ref{enum:pseudodisc},\ref{enum:subset-halfplane}, we have $L \cap S_2 \subset L \cap S_3$.
 Thus, $S_3 \subset S_1 \cup S_2$, which implies that $S_3 \cap X = Y$, i.e., $S_3 \in \cR(Y)$ and the edge $xy$ lies in two triangles of $G$, as desired.
 
 To summarize, $G$ is drawn crossing-free with all vertices on the outer face, and every edge of $G$ lies on the convex hull of $Y'$ or in two triangles.
 This implies that $G$ is maximally outerplanar.
\end{proof}

\subsection{The number of Type~I ranges.}\label{sec:type-I}
 
Recall that for $i=2,\ldots,|X|$ we denote by $\cA_i$ the set of representative halfplanes for $i$-sets of $X$.
In the next two proofs we treat representative halfplanes similarly to $S$-ranges.
Indeed, one can think of a halfplane as a homothet of $S$ with infinitely large dilation and at the same time infinitely large translation (however, formally this is incorrect!).
In the light of Lemma~\ref{lem:next-S-range}, a halfplane defined by $\overline{pq}$ arises as a kind of limit object of a sequences of $S$-ranges in $\cS_{pq}$.
Accordingly, we defined $\n_L(S_1) = L$ if $(S_1 \triangle L) \cap X = \emptyset$.

\begin{proposition}\label{prop:type-I}
 For $k\geq 3$ we have
 \[
  |\cR^k_1| = 2(k-2)|X| - \sum_{i=2}^{k-1}|\cA_i| - (k-1)(k-2).
 \]
\end{proposition}
\begin{proof}
 For a point $p \in X$ and a set $S'$ that is either a Type~I $S$-range or a representative halfplane, we say that $p$ is the \emph{second point} of $S'$ if
 \begin{itemize}
  \item $\partial S' \cap X = \{p,q,r\}$ and the $x$-coordinate of $p$ lies strictly between the $x$-coordinates of $q$ and $r$, or
  \item $\partial S' \cap X = \{p,q\}$ and $S'$ is on the right when going from $p$ to $q$ along $\overline{pq}$.  
 \end{itemize}
 Clearly, every representative halfplane has a unique second point.
 Note that also every Type~I $S$-range has a unique second point, because no two points in $X$ have the same $x$-coordinate.
 Moreover, if $p$ is the second point of $S'$ and $\ell$ denotes the vertical line through $p$, then $S' \cap \ell$ is a vertical segment (if $S'$ is an $S$-range) or ray (if $S'$ is a halfplane) with an endpoint $p$.
 We say that $p$ is a \emph{lower}, respectively \emph{upper}, second point of $S'$ if $S' \cap \ell$ has $p$ as its lower, respectively upper endpoint.

 \medskip

 Now, fix $p \in X$ and the vertical line $\ell$ through $p$.
 Let $L$ and $R$ denote the left and right open halfplanes defined by $\ell$, respectively.
 We want to show that roughly $k-2$ $S$-ranges in $\cR_1^k$ have the lower second point $p$.
 We say $S' \in \cR^k_1 \cup \bigcup_{i \geq 2} \cA_i$ has \emph{property $(a,b)$} if
 \begin{itemize}
  \item $S'$ has the lower second point $p$,
  \item $L \cap S'$ contains exactly $a$ points from $X$, one on $\partial S'$ if $a \geq 1$, and
  \item $R \cap S'$ contains exactly $b$ points from $X$.
 \end{itemize}
 
 \begin{claim}
  Let $m \geq 0$ be the number of points in $X$ whose $x$-coordinate is smaller than the $x$-coordinate of $p$.
  Then for each $a=1,\ldots,\min(m,k-2)$ there exists $S_{a,p} \in \cR^k_1 \cup \bigcup_{i \geq 2} \cA_i$ with property $(a,b)$, so that
  \begin{itemize}
   \item either $a + b \leq k-2$ and $S_{a,p}$ is a halfplane,
   \item or $a + b = k-1$ and $S_{a,p}$ is a Type~I range. 
  \end{itemize}
 \end{claim}
 \begin{claimproof}
  Let $a \in \{1,\ldots,\min(m,k-2)\}$ be fixed.
  We shall first construct $S$-ranges with properties $(0,0), (1,0), \ldots, (a,0)$, respectively, and then $S$-ranges with properties $(a,1), (a,2), \ldots, (a,b)$, respectively.
  We start with any $S$-range $S_0$ with property $(0,0)$.
  Let $q$ denote the upper endpoint of $\ell \cap S_0$.
  We choose $S_0$ so that no four points of $X \cup q$ lie on the boundary of any $S$-range.
  Then we define for $i = 1,\ldots,a$ $S_i$ to be the next $S$-range of $S_{i-1}$ in $L$, i.e., $S_i = \n_L(S_{i-1})$.
  By~\eqref{eq:capture-minus-1},\eqref{eq:capture-plus-1}, $S_i$ has property $(i,0)$.
  In particular, $S_a$ has property $(a,0)$.
  See Figure~\ref{fig:find-S0} for an illustration.
 
 \begin{figure}[htb]
  \centering
  \subfigure[\label{fig:find-S0}]{
   \includegraphics[scale=0.85]{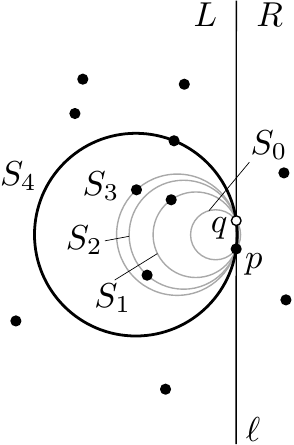}
  }
  \hfill
  \subfigure[\label{fig:find-Ti-cases}]{
   \includegraphics[scale=0.85]{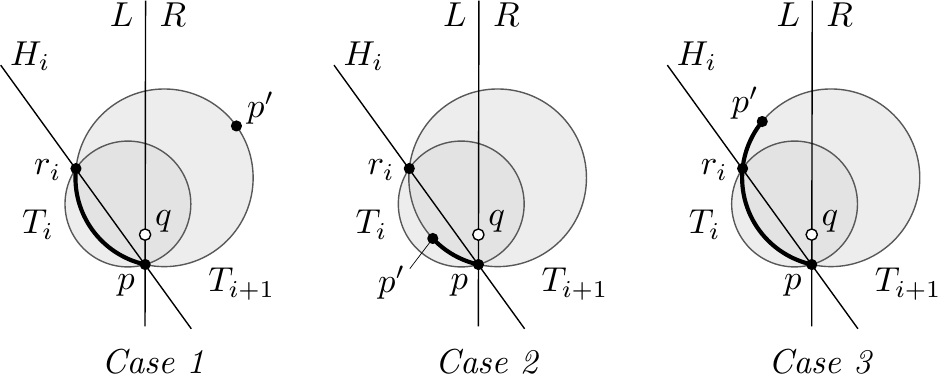}
  }
  \caption{
  \subref{fig:find-S0} The $S$-ranges $S_0,\ldots,S_4$ have property $(0,0),\ldots,(4,0)$, respectively.
  \subref{fig:find-Ti-cases} The three cases in the construction of $T_{i+1}$ based on $T_i$ and $r_i$. The component $\gamma_{i+1}$ of $\partial T_{i+1} \setminus \{p,r_{i+1}\}$ that is completely contained in $L$ is drawn bold.}
 \end{figure}
  
 Next, we shall construct a sequence $T_0,T_1,\ldots T_t$ of $S$-ranges and possibly one halfplane, and a sequence $r_0,r_1,\ldots r_t$ of elements of $X$ with the following properties.
 \begin{enumerate}[label = \textbf{\Alph*)}]
  \item $r_i \in \partial T_i \cap L$ ($i=1,\ldots,t$).
  \item $\ell \cap T_i$ is a segment strictly shorter than $\ell \cap T_{i+1}$ ($i=1,\ldots,t-1$).
  \item $|T_i \cap R \cap X| = x_i$ ($i=1,\ldots,t$) with $x_i \leq x_{i+1}$ ($i=1,\ldots,t-1$).
  \item $|(T_i \setminus \gamma_i) \cap L \cap X| = a$, where $\gamma_i$ denotes the component of $\partial T_i \setminus \{p,r_i\}$ that is completely contained in $L$ ($i=1,\ldots,t$).
  \item When $x_i<x_{i+1}$, then $T_{i+1}$ has property $(a,x_{i+1})$, i.e., $T_{i+1}$ is a Type~I range with the second point $p$ ($i=1,\ldots,t-1$).
  \item $T_t$ is a halfplane.
 \end{enumerate}

 We construct the sequence $T_0,T_1,\ldots,T_t$ as follows.
 Let $T_0 = S_a$, $r_0 \in \partial T_0 \cap (L\cap X)$.
 Assume that $T_0,\ldots, T_i$ and $r_0,\ldots,r_i$ have been constructed and $T_i$ is not a halfplane.
 Let $H_i$ denote the right halfplane defined by $\overline{pr_i}$, i.e., $q \in H_i$.
 We define $T_{i+1} = \n_{H_i}(T_i)$.
 Then by Lemma~\ref{lem:next-S-range} we have $T_{i+1} \cap H_i \supset T_i \cap H_i$ and hence the segment $\ell \cap T_i$ is strictly shorter than $\ell \cap T_{i+1}$.
 Moreover, Lemma~\ref{lem:next-S-range} implies that $|(T_{i+1} \setminus \gamma_{i+1}) \cap L \cap X| = a$ and $x_{i+1} = |T_{i+1} \cap R \cap X| \in \{x_i,x_{i+1}\}$.
   
 Finally, we shall define the point $r_{i+1}$.
 If $T_{i+1} = H_i$ is a halfplane, we set $r_{i+1} = r_i$, $t = i+1$ and the sequence is complete.
 Otherwise, $T_{i+1}$ is a bounded $S$-range, and we consider the unique point $p'$ in $(\partial T_{i+1} \setminus \partial T_i) \cap X$.
 We distinguish three cases.
 
  \begin{enumerate}[itemsep = 5pt, itemindent = 2.2em, label = \textit{Case~\arabic*:}]
   \item \textit{$p' \in R$.} We have that $|T_{i+1} \cap L \cap X| = a$, $|T_{i+1} \cap R \cap X| = x_i + 1$ and $\partial T_{i+1} \cap X = \{r_i,p,p'\}$. So $T_{i+1}$ has property $(a,x_{i+1})$ with $x_{i+1} = x_i + 1$. Set $r_{i+1} = r_i$, which implies $\gamma_{i+1} \cap X = \emptyset$.
   
   \item \textit{$p' \in L \setminus H_i$.}  Then $|T_{i+1} \cap L \cap X| = a$ and $|T_{i+1} \cap R \cap X| = x_i$, just like $T_i$. In this case we set $r_{i+1} = p'$, which gives again $\gamma_{i+1} \cap X = \emptyset$.
 
   \item \textit{$p' \in L\cap H_i$.} Then $|T_{i+1} \cap L \cap X| = a+1$ and $|T_{i+1} \cap R \cap X| = x_i$. We set $r_{i+1} = p'$, which implies $r_i \in \gamma_{i+1}$ and hence $|(T_{i+1} \setminus \gamma_{i+1}) \cap L \cap X| = a$. 
 \end{enumerate} 
 
 We refer to Figure~\ref{fig:find-Ti-cases} for an illustration.
 We see that if $T_{i+1}$ is not a halfplane, we either have $r_{i+1} \neq r_i$ or $x_{i+1} > x_i$.
 Since there are finitely many possibilities for $r_i$ and $x_i$ and no pair $\{r_i,x_i\}$ occurs twice, at some point $T_{i+1}$ is a halfplane.
 
 Note that $T_t$ is a halfplane with property $(a,x_t)$.
 Hence, if $x_t < k-1-a$, then $T_t$ is the desired $S$-range $S_{a,p} \in \bigcup_{i \geq 2} \cA_i$.
 Otherwise, a subsequence of $T_0,T_1,\ldots T_t$ consists of Type~I ranges with properties $(a,0),(a,1),\ldots,(a,b)$ with $a+b = k-1$ and the last element of this subsequence is the desired $S$-range $S_{a,p} \in \cR^k_1$.
 This completes the proof of the claim.
 \end{claimproof}

 Note that for every $S' \in \cR^k_1 \cup \bigcup_{i \geq 2} \cA_i$ we have $S' = S_{a,p}$ for at most one pair $(a,p)$ of a number $a \in \{1,\ldots,k-2\}$ and a point $p \in X$.
 The above claim states that for given $(a,p)$ we find $S_{a,p}$, unless fewer than $a$ points in $X$ have smaller $x$-coordinate than $p$.
 This rules out $\binom{k-1}{2}$ of the $(k-2)|X|$ pairs $(a,p)$ and we conclude that
 \begin{equation}
  (k-2)|X| - \binom{k-1}{2} \leq |\cR_1^{k,\downarrow}| + \sum_{i=2}^{k-1}|\cA^{\downarrow}_i|,\label{eq:lower-middle-point}
 \end{equation}
 where $\cR_1^{k,\downarrow} \subseteq \cR_1^k$, respectively $\cA^{\downarrow}_i \subseteq \cA_i$ ($i=2,\ldots,k-2$), are the subsets of Type~I ranges, respectively representative halfplanes, with a \emph{lower} second point.
 
 By symmetry, we obtain an analogous inequality for the Type~I ranges $\cR^{k,\uparrow}_1$ and representative halfplanes $\cA^\uparrow_i$ with an \emph{upper} second point.
 Together with $\cR_1^{k,\downarrow} \cap \cR_1^{k,\uparrow} = \emptyset$ and $\cA^{\downarrow}_i \cap \cA^{\uparrow}_i = \emptyset$ for $i \geq 2$ this shows that
 \begin{equation}
  2(k-2)|X|-(k-1)(k-2) \leq |\cR_1^{k,\downarrow}| + |\cR_1^{k,\uparrow}| + \sum_{i=2}^{k-1}(|\cA^{\downarrow}_i| + |\cA^{\uparrow}_i|) \leq |\cR^k_1| + \sum_{i=2}^{k-1}|\cA_i|.\label{eq:lower-and-upper}
 \end{equation}
 In order to finish the proof of Proposition~\ref{prop:type-I} it remains to prove that~\eqref{eq:lower-middle-point} (and hence also~\eqref{eq:lower-and-upper}) holds with equality.
 For this, we show that for every point $p \in X$ and every $a \in \{1,\ldots,k-2\}$ there is at most one $S$-range in $\cR_1^k$ with property $(a,b)$ for $a+b = k-1$ and at most one halfplane in $\bigcup_{i \geq 2} \cA_i$ with property $(a,b)$ for $a + b \leq k-2$.
 If $S_1,S_2 \in \cR_1^k$ are two distinct $S$-ranges with the same lower second point $p$, then by Lemma~\ref{lem:small-facts}~\ref{enum:subset-halfplane} $S_1 \cap L \subset S_2 \cap L$ or $S_2 \cap L \subset S_1 \cap L$.
 Observe that, if $S_1,S_2 \in \bigcup_{i\geq 2} \cA_i$ are distinct closed halfplanes with the same lower second point $p$, then we also have $S_1 \cap L \subset S_2 \cap L$ or $S_2 \cap L \subset S_1 \cap L$.
 So in either case we may assume that $S_1 \cap L \subset S_2 \cap L$.
 Now, if $S_2$ has property $(a,b)$, then $\partial S_2 \cap L \cap X \neq \emptyset$ and hence $|S_2 \cap L \cap X| > |S_1 \cap L \cap X|$.
 Thus $S_1$ and $S_2$ can not both have property $(a,b)$ for the same $a$.
 
 We conclude that inequality~\eqref{eq:lower-middle-point} holds with equality.
 Thus~\eqref{eq:lower-and-upper} also holds with equality, which is the statement of Proposition~\ref{prop:type-I}.
\end{proof}

\subsection{Relation between the number of Type~I and Type~II ranges.}\label{sec:exact-bound}

Recall that for a fixed Type~I hyperedge $Y$ in $\cH(X,S,k)$ we denote by $\cR(Y)$ the set of representative ranges for $Y$. 

\begin{proposition}\label{prop:extensions}
 For $k \geq 3$ we have
 \[
  3|\cR^k_1| + 2|\cR^k_2| = 3|\cR^{k+1}_1| + |\cA_k| + 2\sum_{Y \in \cE^k_1}(|\cR(Y)|-1).
 \]
\end{proposition}
\begin{proof}
 Consider the set $P$ of all ordered pairs $(S_1,S_2)$ of an $S$-range $S_1$ and an $S$-range or representative halfplane $S_2$, such that
 \begin{enumerate}[label = (\Alph*)]
  \item $S_1 \in \cR^k_1 \cup \cR^k_2$.\label{enum:origin}
  \item $\partial S_1 \cap \partial S_2$ is a pair $p,q$ of points in $X$.\label{enum:boundary-points}
  \item $S_2 = \n_L(S_1)$, where $L$ is one of the halfplanes defined by $\overline{pq}$.\label{enum:next}
  \item $X \cap S_1 \subseteq X \cap S_2$.\label{enum:subset}
 \end{enumerate}
 For a pair $(S_1,S_2) \in P$ we say that $S_2$ \emph{is an image of} $S_1$ and $S_1$ \emph{is a preimage of} $S_2$.
 Note that, if $S_2$ is an image of $S_1$, then $S_1$ contains $k$ points from $X$ and thus by~\eqref{eq:capture-minus-1},\eqref{eq:capture-plus-1} $S_2$ contains either $k$ or $k+1$ points from $X$.
 In the former case, $S_1$ and $S_2$ are either distinct representative $S$-ranges for the same Type~I hyperedge in $\cH(X,S,k)$ or $S_2$ is a halfplane, see Figure~\ref{fig:type-I-extensions}, while in the latter case $S_2$ is a Type~I range in $\cH(X,S,k+1)$, see Figure~\ref{fig:type-I-extensions} and~\subref{fig:type-II-extensions}.

 \begin{figure}[htb]
  \centering
  \subfigure[\label{fig:type-I-extensions}]{
   \includegraphics[scale=0.85]{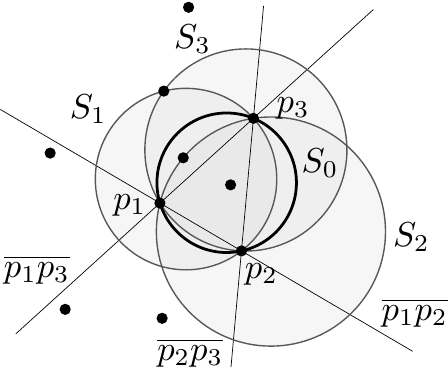}
  }
  \hspace{0.5em}
  \subfigure[\label{fig:type-II-extensions}]{
   \includegraphics[scale=0.85]{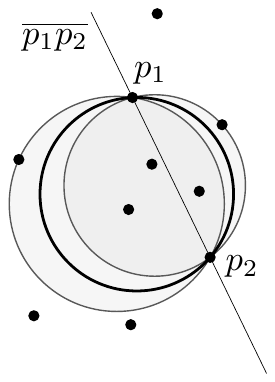}
  }
  \hspace{0.5em}
  \subfigure[\label{fig:same-hyperedge}]{
   \includegraphics[scale=0.85]{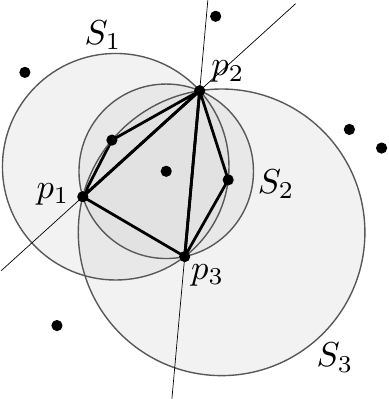}
  }
  \caption{\subref{fig:type-I-extensions} The three images $S_1,S_2,S_3$ of a Type~I range $S_0$ (in bold). Note that $S_1$ and $S_3$ correspond to the same hyperedge in $\cH(X,S,k+1)$, and that $S_0$ and $S_2$ correspond to the same hyperedge in $\cH(X,S,k)$. \subref{fig:type-II-extensions} The two images of a Type~II range (in bold); one being also an image of the Type~I range in~\subref{fig:type-I-extensions}. \subref{fig:same-hyperedge} Three representative ranges for the same hyperedge $Y$ and the outerplanar graph $G(Y)$ (in bold). Here $(S_1,S_2) \in P_4$ and $(S_2,S_1) \in P_4$, both with respect to $\{p_2,p_3\}$, as well as $(S_2,S_3) \in P_4$ and $(S_3,S_2) \in P_4$, both with respect to $\{p_1,p_2\}$.}
 \end{figure}
 
 We partition $P$ in two different ways; once with respect to the possibilities for preimages and once with respect to the possibilities for preimages.
 Firstly, $P = P_1 \dot\cup P_2$, where $P_1$ and $P_2$ contain all pairs $(S_1,S_2)$ with $S_1 \in \cR^k_1$ and $S_1 \in \cR^k_2$, respectively.
 Secondly, $P = P_3 \dot\cup P_4 \dot\cup P_5$, where $P_3$, $P_4$ and $P_5$ contain all pairs $(S_1,S_2)$ with $S_2 \in \cR_1^{k+1}$, $S_2 \in \cA_k$ and $S_2 \in \cR_1^k$, respectively.
 We summarize:
 \begin{align*}
  P_i &= \{(S_1,S_2) \in P \mid S_1 \in \cR_i^k\}, \quad i=1,2,\\
  P_3 &= \{(S_1,S_2) \in P \mid S_2 \in \cR_1^{k+1}\},\\
  P_4 &= \{(S_1,S_2) \in P \mid S_2 \in \cA_k\},\\
  P_5 &= \bigcup_{Y\in \cE_1^k} \{(S_1,S_2)\in P \mid S_1,S_2 \in \cR(Y)\},\\
  P   &= P_1\dot\cup P_2  \quad \text{and} \quad P=P_3\dot\cup P_4 \dot\cup P_5.
 \end{align*}
 We shall show that, on one hand, $|P_1|= 3|\cR^k_1|$ and $|P_2| =  2|\cR^k_2|$, while on the other hand, $|P_3| = 3|\cR^{k+1}_1|$, $|P_4| = |\cA_k|$ and $|P_5| = 2\sum_{Y \in \cE^k_1}(|\cR(Y)|-1)$.
 Together with $|P_1| + |P_2| = |P| = |P_3| + |P_4| + |P_5|$ this will conclude the proof.
 
 \medskip
 
 To prove that $|P_1|= 3|\cR^k_1|$, consider $S_1 \in \cR^k_1$ and let $\partial S_1 \cap X = \{p_0,p_1,p_2\}$.
 For $i=0,1,2$ let $H_i$ be the halfplane defined by $\overline{p_{i-1}p_{i+1}}$ containing $p_i$, where indices are taken modulo~$3$.
 By~\ref{enum:boundary-points},\ref{enum:next} and~\ref{enum:subset}, every image $S_2$ of $S_1$ is the next $S$-range of $S_1$ in $H_i$ for some $i \in \{0,1,2\}$, whose existence and uniqueness is given by Lemma~\ref{lem:next-S-range}~\ref{enum:next-exists}.
 As every such next $S$-range contains $X \cap S_1$, $S_1$ has exactly three images.
 
 \medskip
  
 To prove that $|P_2| =  2|\cR^k_2|$, consider $S_1 \in \cR^k_2$ and let $X \cap \partial S_1 = \{p_1,p_2\}$.
 If $S_2$ is an image of $S_1$, then by~\ref{enum:boundary-points},\ref{enum:next} $S_2 = \n_H(S_1)$, where $H$ is a halfplane defined by $\overline{p_1p_2}$.
 As $\partial S_1 \cap X = \{p,q\}$, the next $S$-range in either halfplane contains $S_1 \cap X$, i.e.,~\ref{enum:subset} is satisfied.
 Hence $S_1$ has two images, one for each halfplane.
  
 \medskip
  
 To prove that $|P_3| = 3|\cR^{k+1}_1|$, consider any $S_2 \in \cR^{k+1}_1$ and let $X \cap \partial S_2 = \{p_0,p_1,p_2\}$.
 Let $H_i$ and $\bar{H}_i$ be the halfplanes defined by $\overline{p_{i-1}p_{i+1}}$ containing $p_i$ and \emph{not} containing $p_i$, respectively, where indices are taken modulo~$3$ again.
 By~\ref{enum:boundary-points},~\ref{enum:next} and~\ref{enum:subset} every preimage $S_1$ of $S_2$ corresponds to a point $p_i \in \{p_0,p_1,p_2\}$ with $p_i \in S_2 \setminus S_1$, such that $S_2 = \n_{H_i}(S_1)$.
 Indeed, if $S' = \n_{\bar{H_i}}(S_2)$ captures $k$ points from $X$, then $(S',S_2) \in P_1$.
 Whereas, if $S'$ captures $k+1$ points from $X$, then $Y = X \cap S' \cap S_2$ is a Type~II hyperedge and for its representative range $S''$ we have $(S'',S_2) \in P_2$.
 Finally, $S'$ can not capture $k+2$ points since $p_i \in S_2 \setminus S'$.
 Hence $S_2$ has exactly three preimages.
  
 \medskip
 
 To prove that $|P_4| = |\cA_k|$, we need to show that every halfplane $H \in \cA_k$ is the image of exactly one $S$-range in $\cR_1^k \cup \cR_2^k$.
 In fact, if $\{p,q\} = X \cap \partial H$ and $Y = H \cap X$, then consider $S$-ranges defined by $p,q$ and a third point from $Y$.
 By Lemma~\ref{lem:small-facts} these $S$-ranges are well-defined and by Lemma~\ref{lem:next-S-range} they are linearly ordered by inclusion in $H$.
 For the $S$-range $S'$ with $S' \cap H$ being inclusion-maximal we have $(S',H) \in P_4$, as desired.
 
 \medskip
 
 Finally, we prove that $|P_5| = 2\sum_{Y \in \cE^k_1}(|\cR(Y)|-1)$.
 Consider any hyperedge $Y \in \cE^k_1$ and the graph $G(Y)$ defined above, whose edges are all pairs $\{p,q\}\subseteq Y$ such that $p,q\in \partial S'$ for a representative $S$-range $S' \in \cR(Y) \subset \cR^{k}_1$. 
 By Lemma~\ref{lem:outerplanar-graph}, connecting any two points in $Y$ that are adjacent in $G(Y)$ with a straight line segment gives a maximally outerplanar drawing of $G(Y)$.
 Now if $(S_1,S_2) \in P_5$, then $\partial S_1 \cap \partial S_2 = \{p,q\}$ is an inner edge of $G(Y)$, see Figure~\ref{fig:same-hyperedge}.
 Moreover, exactly two $S$-ranges in $\cR(Y)$ have $p$ and $q$ on their boundary, because $S$-ranges in $\cR(Y)$ correspond to triangles in $G(Y)$ and $G(Y)$ is maximally outerplanar.

 Thus, every pair $(S_1,S_2) \in P_5$ with $S_2 \in \cR_1^k$ gives rise to an inner edge of $G(Y)$ and every inner edge $\{p,q\}$ of $G(Y)$ gives exactly two such ordered pairs in $P_5$.
 Because a maximally outerplanar graph with $|\cR(Y)|$ triangles has $|\cR(Y)|-1$ inner edges, we have the desired equality.
  
 \medskip
 
 Now we conclude that $|P| = |P_1|+|P_2| = 3|\cR^k_1| + 2|\cR^k_2|$, whereas $|P| = |P_3|+|P_4|+|P_5| = 3|R^{k+1}_1| + |\cA_k| + 2\sum_{Y \in \cE^k_1}(|\cR(Y)|-1)$.
 Together this gives the claimed equality.
\end{proof}

\section{Proof of Theorem~\ref{thm:main}.}\label{sec:upper-bound}

\begin{proof}[Proof of Theorem~\ref{thm:main}]
 For $k = 1$ and any $X$, the hypergraph $\cH(X,S,k)$ clearly has $|X| = (2k-1)|X| - k^2 + 1 - \sum_{i=1}^{k-1}a_i$ hyperedges.
 For $k = 2$ and any $X$, $\cH(X,S,k)$ is the well-known Delaunay triangulation $D_S$ with respect to the convex distance function defined by $S$.
 In particular, every inner face of $D_S$ is a triangle and the outer face is the convex hull of $X$, i.e., has length $a_1$.
 Thus, by Euler's formula the number of hyperedges of $\cH(X,S,2)$ (edges of $D_S$) is given by $3|X|-3-a_1 = (2k-1)|X| - k^2 + 1 - \sum_{i=1}^{k-1}a_i$.
 So we may assume that $k \geq 3$.
 Moreover, by Lemma~\ref{lem:general-position} we may assume that $S$ is nice and $X$ is in general position with respect to $S$.
 
 By Proposition~\ref{prop:extensions} we have
 \begin{equation}
  3|\cR^k_1| + 2|\cR^k_2| = 3|\cR^{k+1}_1| + |\cA_k| + 2\sum_{Y \in \cE^k_1}(|\cR(Y)|-1).\label{eq:extensions}
 \end{equation}
 By Proposition~\ref{prop:type-I} we have
 \begin{align}
  |\cR^k_1| &= 2(k-2)|X| - \sum_{i=2}^{k-1} |\cA_i| - (k-1)(k-2), \mbox{ and} \label{eq:typeI-UB}\\
  |\cR^{k+1}_1| &= 2(k-1)|X| - \sum_{i=2}^k |\cA_i| - k(k-1).\label{eq:typeI-LB}
 \end{align}
 
 Putting~\eqref{eq:clearly-1},~\eqref{eq:clearly-2},~\eqref{eq:extensions},~\eqref{eq:typeI-UB} and~\eqref{eq:typeI-LB} together, we conclude that
 \begin{align*}
  2(|\cE^k_1| + |\cE^k_2|) \hspace{7pt}&\hspace{-7pt}\overset{\eqref{eq:clearly-1},\eqref{eq:clearly-2}}{=} 2|\cR^k_1| -2\sum_{Y \in \cE^k_1}(|\cR(Y)|-1) + 2|\cR^k_2|\\[5pt]
   &\overset{\eqref{eq:extensions}}{=} 3|\cR^{k+1}_1| + |\cA_k| - |\cR^k_1|\\[5pt]
   &\hspace{-7pt}\overset{\eqref{eq:typeI-UB},\eqref{eq:typeI-LB}}{=} 6(k-1)|X| - 3\sum_{i=2}^k |\cA_i| - 3k(k-1) + |\cA_k|\\
   &\hspace{2em} - 2(k-2)|X| + \sum_{i=2}^{k-1}|\cA_i| + (k-1)(k-2)\\[5pt]
   &= 2\left((2k-1)|X| - k^2 + 1 - \sum_{i=2}^k|\cA_i|\right).
 \end{align*}
 Thus we have with Lemma~\ref{lem:halfplanes} that $|\cE^k_1| + |\cE^k_2| = (2k-1)|X| - k^2 + 1 - \sum_{i=2}^k|\cA_i| = (2k-1)|X| - k^2 + 1 - \sum_{i=1}^{k-1}a_i$, as desired.
\end{proof}

\section{Conclusions and remarks}

In this paper we investigated $k$-uniform hypergraphs whose vertex set $X$ is a set of points in the plane and whose hyperedges are exactly those $k$-subsets of $X$ that can be captured by a homothetic copy of a fixed convex compact set $S$.
These are so called $k$-uniform $S$-capturing hypergraphs.
We have shown that every such hypergraph has at most $(2k-1)|X| -k^2 + 1 -\sum_{i=1}^{k-1} a_i$ hyperedges and that this is tight for every nice shape $S$.
Here $a_i$ denotes the number of $i$-subsets of $X$ that can be separated with a straight line.

\medskip

\noindent
As an immediate corollary we obtain that if $S$ is nice, then the total number of subsets of $X$ captured by some homothet of $S$ is given by the cake number $\binom{|X|}{3} + \binom{|X|}{2} + \binom{|X|}{1}$.
Moreover, we obtain a bound on the number of hyperedges of size \emph{at most $k$}:
For every point set $X$, every convex set $S$ and every $k \geq 1$ at most $k^2|X|$ different non-empty subsets of $X$ of size at most $k$ can be captured by a homothetic copy of $S$.
This refines the recent $O(k^2|X|)$ bound by Buzaglo, Pinchasi and Rote~\cite{BuPiRo13}.

Another interesting open problem concerns topological hypergraphs defined by a family of pseudodiscs.
Here, the vertex set $X$ is again a finite point set in the plane and every hyperedge is a subset of $X$ surrounded by a closed Jordan curve such that any two such curves have at most two points in common.
Buzaglo, Pinchasi and Rote~\cite{BuPiRo13} prove that every pseudodisc topological hypergraph has at most $O(k^2|X|)$ hyperedges of size \emph{at most $k$}.

\begin{problem}
 What is the maximum number of hyperedges of size \emph{exactly} $k$ in a pseudodisc topological hypergraph?
\end{problem}

As we learned after submission, Chevallier~\textit{et al.}~\cite{CFSS14,CFSS15} have an unpublished manuscript in which they prove that every inclusion-maximal $k$-uniform \emph{convex} pseudodisc topological hypergraph with $n$ vertices has exactly $(2k-1)|X| -k^2 + 1 -\sum_{i=1}^{k-1} a_i$ hyperedges.
This independent result implies our Theorem~\ref{thm:main}.
However, their proof is $40$ pages long and involves higher-order Voronoi diagrams and higher-order centroid Delaunay triangulations.
Our proof uses a completely different technique; it is short and completely self-contained, and hence is of independent interest.

\section*{Acknowledgments}

We would like to thank G\"unter Rote and J\'anos Pach for helpful remarks and referring us first to~\cite{BuPiRo13} and later on to~\cite{CFSS14}.

\bibliographystyle{abbrv}
\bibliography{lit}

\section*{Appendix}

For a set $S$, we call a line $\ell$ \emph{touching} $S$ if $\ell$ intersects $S$ in exactly one point.
So, $S$ is nice if and only if for each point on its boundary there is exactly one touching line touching $S$ at this point.

\setcounter{theorem}{5}
\begin{lemma}
 If $S$ is a nice shape, $S_1$ and $S_2$ are distinct $S$-ranges, then each of the following holds.
 \begin{enumerate}[label = (\roman*)]
  \item  $\partial S_1 \cap \partial S_2$ is a set of at most two points.
  
  \item If $\partial S_1 \cap \partial S_2 = \{p,q\}$ and $L$ and $R$ are the two open halfplanes defined by $\overline{pq}$, then
   \begin{itemize}
    \item $S_1 \cap L \subset S_2 \cap L$ and $S_1 \cap R \supset S_2 \cap R$ or
    \item $S_1 \cap L \supset S_2 \cap L$ and $S_1 \cap R \subset S_2 \cap R$.
   \end{itemize}
  
  \item Any three non-collinear points lie on the boundary of a unique $S$-range.
    
  \item For a subset of points $X \subset \Real^2$ and any $Y \subset X$, $|Y| \geq 2$, that is captured by some $S$-range there exists at least one $S$-range $S'$ with $Y = X \cap S'$ and $|\partial S' \cap X| \geq 2$.
 \end{enumerate}
\end{lemma}

\begin{proof}
 \begin{enumerate}[label = (\roman*)]
  \item We show that for any two $S$-ranges $S_1$, $S_2$ such that $\{p, q, r\} \subseteq \partial S_1 \cap \partial S_2$, for distinct $p, q, r$, $S_1$ coincides with $S_2$.
  Assume not, and consider homothetic maps $f_1$, $f_2$ from $S_1$, $S_2$ to $S$, respectively.
  Let $p_i, q_i, r_i$ be the images of $p, q, r$ under $f_i$, $i=1,2$.
  Then we have that $p_1, q_1, r_1$ and $p_2, q_2, r_2$ form congruent triangles $T_1$, $T_2$ with vertices on $\partial S$.
  If these two triangles coincide, then $S_1=S_2$.
  Otherwise consider two cases: a corner of one triangle is contained in the interior of the other triangle or not.
  If, without loss of generality, a corner of $T_1$ is in the interior of $T_2$, then by convexity, this corner can not be on $\partial S$.
  Otherwise, $T_1$ and $T_2$ are either disjoint or share a point on the corresponding side.
  In either case, one of the sides of $T_1$ is on the same line as the corresponding side of $T_2$, otherwise convexity of $S$ is violated.
  Then, the boundary of $S$ contains $3$ collinear points, a contradiction to the fact that $S$ is a nice shape.
  
  \item We have, without loss of generality that $S_1 \cap L \subset S_2 \cap L$.
  If $S_1 \cap R \supset S_2 \cap R$, we are done.
  Otherwise, we have that $S_1 \cap R \subset S_2 \cap R$, and, in particular $S_1\subseteq S_2$.
  Note that at $p$ and $q$ the $S$-ranges $S_1$ and $S_2$ have the same touching lines.
  Indeed, these lines are unique since $S$ is nice.
  Consider maps $f_1$ and $f_2$ as before.
  We see that $p,q$ are mapped into the same pair of points under both maps.
  Thus $S_1 = S_2$, a contradiction.
 \end{enumerate}

 \medskip
 
 \noindent
 The remaining two items can be proven by considering two points $p,q$ in the plane and the set $\cS_{pq}$ of all $S$-ranges $S'$ with $p,q \in \partial S'$.
 Indeed, given fixed $p$ and $q$ there is a bijection $\phi$ between the $S$-ranges in $\cS(p,q)$ and the set $\mathscr{L}(p,q)$ of lines whose intersection with $S$ is a non-trivial line segment parallel to the line $\overline{pq}$.
 We refer to Figure~\ref{fig:lines-to-ranges} for an illustration.
 
 \begin{figure}[htb]
  \centering
  \subfigure[\label{fig:lines-to-ranges}]{
   \includegraphics{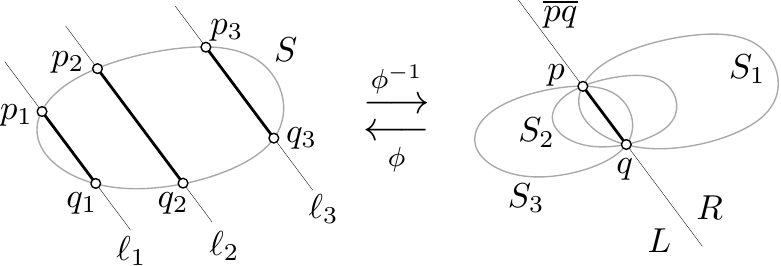}
  }
  \hfill
  \subfigure[\label{fig:small-facts-1}]{
   \includegraphics{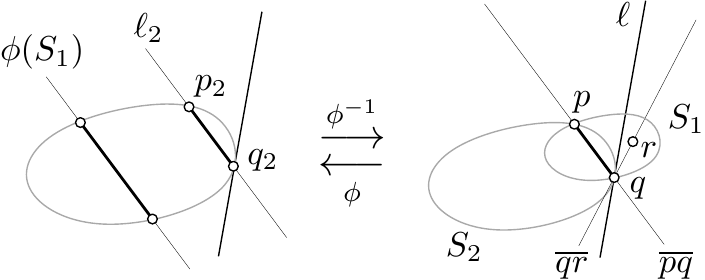}
  }
  \caption{\subref{fig:lines-to-ranges} A nice shape $S$, two points $p,q$ in the plane, three lines $\ell_1,\ell_2,\ell_3 \in \mathscr{L}(p,q)$ and the corresponding $S$-ranges $S_1,S_2,S_3 \in \cS_{pq}$. \subref{fig:small-facts-1} Given an $S$-range $S_1 \in \cS_{pq}$ with $r \in S_1$ we can find an $S$-range $S_2 \in \cS_{pq}$ with $r \notin S_2$.}
 \end{figure}
 
 We verify that this bijection exists:
 For a line $\ell' \in \mathscr{L}(p,q)$ such that $\ell \cap \partial S = \{p_1, q_1\}$, let $S'$ be an $S$-range obtained by contracting and translating $S$ such that $p_1$ and $q_1$ are mapped into $p$ and $q$, respectively.
 Let $\phi^{-1}(\ell)=S'$.
 Given an $S'$-range with $p,q \in \partial S$, consider a translation and contraction that maps $S'$ to $S$.
 Let $p_1$ and $q_1$ be the images of $p$ and $q$ under this transform, then let $\phi(S')=\ell$, where $\ell$ is a line through $p_1$ and $q_1$.
 See again Figure~\ref{fig:lines-to-ranges} for an illustration.

 \begin{enumerate}[label = (\roman*), start = 3]
  \item Consider any three non-collinear points $p,q,r$ in the plane. We shall first show that there is an $S$-range with $p,q,r$ on its boundary. Let $S_1$ be an $S$-range of smallest area containing all three points. Clearly, $|\partial S_1 \cap \{p,q,r\}|$ is either $2$ or $3$. In the latter case we are done. So assume without loss of generality that $\partial S_1 \cap \{p,q,r\} = \{p,q\}$. 
  
  Now, we claim that there is another $S$-range $S_2$ that contains $p$ and $q$ but not $r$, with $p,q \in \partial S_2$. To find $S_2$, containing $p,q$ on its boundary and not containing $r$, consider the triangle $p,q,r$ and a line $\ell$ that goes through $q$ having $p$ and $r$ on the different sides. 
  Assume without loss of generality that $q$ is the lowest point among $p,q,r$, and that $\overline{qr}$ and $\ell$ have positive slopes. 
  Next, let $q_2$ be a point on $\partial S$ whose touching line is parallel to $\ell$ and such that $S$ is above this touching line. See Figure~\ref{fig:small-facts-1}.
  Let $\ell_2$ be the line parallel to $\overline{pq}$ and containing $q_2$ and let $p_2$ be the unique point in $\partial S \cap \ell_2$ different from $q_2$. Note that $|\ell_2 \cap \partial S| = 2$ follows from the fact that $S$ is a nice shape and $\ell_2$ and $\ell$ have different slopes.
  Let $S_2 = \phi^{-1}(\ell_2)$. We see that $S_2$ is above $\ell$, but $r$ is below $\ell$. So, $r\notin S_2$. 
    
  Since $\phi$ is continuous and any two lines in $\mathscr{L}(p,q)$ can be continuously transformed into each other within $\mathscr{L}(p,q)$, we conclude that $S_1$ can be continuously transformed into $S_2$ within $\cS_{pq}$. Thus, there is an $S$-range $S' \in  \cS_{pq}$ such that $r\in \partial S'$. This proves that any three non-collinear points $p,q,r$ lie on the boundary of some $S$-range. The uniqueness of such a range follows from~\ref{enum:pseudodisc}.
 
  \item Let $Y$ be  a hyperedge in $\cH(X,S,k)$ and $S_1$ be an $S$-range capturing $Y$.
  Contract $S_1$ until the resulting range $S_2$  contains at least one point, $p$ of $Y$, 
  on its boundary. If $|\partial S_2 \cap Y| \geq 2$, we are done. 
  Otherwise, consider a small $S$-range $S_3$ containing $p$ on its boundary and not containing any other points of $X$. Let $q$ be the second point in $\partial S_2 \cap \partial S_3$. 
  Similarly to the previous argumentation, $S_2$ can be continuously transformed into $S_3$ within $ \cS_{pq}$. Each intermediate $S$-range is contained in $S_2\cup S_3$ and thus contains no points of $X\setminus Y$. One of the intermediate $S$-ranges will contain another point of $Y$ on its boundary.
 \end{enumerate} 
\end{proof}

\begin{lemma}
 For any point set $X$, positive integer $k$ and a convex compact set $S$, there is a nice shape $S'$ and a point set $X'$ in general position with respect to $S'$, such that $|X'|=|X|$ and the number of edges in $\cH(X',S',k)$ is at least as large as the number of edges in $\cH(X,S,k)$. 
\end{lemma}

\begin{proof}
 First, we shall modify $S$ slightly.
 Consider all hyperedges of $\cH(X,S,k)$ and for each of them choose a single capturing $S$-range.
 Recall that two $S$-ranges are contraction/dilation of one-another if in their corresponding homothetic maps the origin is mapped to the same point.
 For each hyperedge consider two distinct capturing $S$-ranges $S_1,S_2$ with $S_2$ being a dilation of $S_1$.
 Among all hyperedges, consider the one for which these two $S$-ranges, $S_1 \subset S_2$ are such that the stretching factor between $S_1$ and $S_2$ is the smallest.

 Let $S'$ be a nice shape, $S_1 \subset S' \subset S_2$. Replace each of the other capturing ranges 
 with an appropriate $S'$-range.
 Now, we have that $\cH(X,S',k)$ has at least at many hyperedges as $\cH(X,S,k)$.
 Next, we shall move the points from $X$ slightly so that the new set $X'$ is in a general position with respect to $S'$ and contains as many hyperedges as $\cH(X,S',k)$.
  
 Observe first that since $X$ is a finite point set, we can move each point of $X$ by some small distance, call it $\epsilon$ in any direction such that the resulting hypergraph has the same set of hyperedges as  $\cH(X,S',k)$.

 Call a point $x\in X$ \emph{bad} if either $x$ is on a vertical line together with some other point of $X$,  $x$ is on a line with two other points of $X$, or $x$ is on the boundary of an $S'$-range together with at least three other vertices of $X$.
 
 We shall move a bad $x$ such that a new point set has smaller number of bad points and such that the resulting hypergraph has at least as many edges as $\cH(X,S',k)$.
 From a ball $B(x,\epsilon)$ delete all vertical lines passing through a point of $X$, delete all lines that pass through at least two points of $X$ and delete all boundaries of all $S'$-ranges containing at least $3$ points of $X$.
 All together we have deleted at most $n + \binom{n}{2} + \binom{n}{3}$, where $n = |X|$, curves because there is one vertical line passing through each point, at most $\binom{n}{2}$ lines passing through some two points of $X$ and at most $\binom{n}{3}$ $S'$-ranges having some three points of $X$ on their boundary.
 So, there are points left in $B(x, \epsilon)$ after this deletion.
 Replace $x$ with an available point $x'$ in $B(x,\epsilon)$.
 Observe that $x'$ is no longer bad in a new set $X- \{x\}\cup \{x'\}$.
 Moreover, if $z\in X$, $z\neq x$ was not a bad point, it is not a bad point in a new set $X- \{x\}\cup \{x'\}$.
 Indeed, since $x'$ is not on a vertical line with any other point of $X$ and not on any line containing two points of $X$, $z$ is not on a bad line with $x'$.
 Moreover, since $x'$ is not on the boundary of an $S'$-range together with at least three other points of $X$, $z$ can not be together with $x'$ on the boundary of an $S'$-range containing more than $3$ points of $X$ on its boundary.
\end{proof}

\begin{lemma}
 Let $p,q$ be two points such that no four points in $X \cup \{p,q\}$ lie on the boundary of an $S$-range.
 Let $L$ be a halfplane defined by $\overline{pq}$.
 Then the following holds.
 \begin{enumerate}[label = (\roman*)]
  \item The $S$-ranges in $\cS_{pq}$ are linearly ordered, denoted by $\prec_{pq}$, by inclusion of their intersection with $L$:
  \[
   S_1 \prec_{pq} S_2 \quad \Leftrightarrow \quad S_1 \cap L \subset S_2 \cap L \qquad \text{for all } S_1,S_2 \in \cS_{pq}
  \]
  \item For each $S_1 \in \cS_{pq}$ there exists a $\prec_{pq}$-minimal $S_2 \in \cS_{pq}$ with
  \[
   (\partial S_2 \setminus \partial S_1) \cap X \neq \emptyset \qquad \text{and} \qquad S_1 \prec_{pq} S_2
  \]
  if and only if $S_1 \triangle L$ contains a point from $X$ in its interior.
 \end{enumerate}
\end{lemma}
 
\begin{proof}
 \begin{enumerate}[label = (\roman*)]
  \item This follows immediately from Lemma~\ref{lem:small-facts}~\ref{enum:subset-halfplane}.

  \item First assume that $(S_1 \triangle L) \cap X \neq \emptyset$.
   For each point $r \in (S_1 \triangle L) \cap X$ consider the $S$-range $S_r$ with $\{p,q,r\} \subset \partial S_r$.
   The existence and uniqueness of $S_r$ is given by Lemma~\ref{lem:small-facts}~\ref{enum:unique-range}.
   By the first item, the $S$-ranges in $\{S_r \mid r \in (S_1 \triangle L) \cap X\}$ are linearly ordered by inclusion of their intersection with $L$.
   Hence there exists an $S$-range $S_2$ in this set, which is $\prec_{pq}$-minimal.
   
   Let $r$ be the point in $\partial S_2$ different from $p$ and $q$.
   If $r \in L$, then $r \notin S_1$ and thus $L \cap S_1 \subset L \cap S_2$.
   If $r \notin L$, then $r \in S_1 \setminus \partial S_1$ and thus $S_2 \cap R \subset S_1 \cap R$, where $R$ is the other halfplane defined by $\overline{pq}$.
   In any case we have $S_1 \prec_{pq} S_2$.
   Moreover, $r \in (\partial S_2 \setminus \partial S_1) \cap X$ and hence $S_2$ is the desired $S$-range.

   \medskip

   Now assume that $S_2$ is a $\prec_{pq}$-minimal $S$-range in $\cS_{pq}$ with $(\partial S_2 \setminus \partial S_1) \cap X \neq \emptyset$ and $S_1 \prec_{pq} S_2$.
   We claim that the point $r$ in $(\partial S_2 \setminus \partial S_1) \cap X$ lies in the interior of $S_1 \triangle L$.
   If $r \in R$, then $r \in S_1$, because $\partial S_2 \cap R \subset S_1 \cap R$.
   On the other hand, if $r \in L$, then $r \notin S_1$, because $\partial S_1 \cap L \subset S_2 \cap L$.
   Hence $r$ lies in the interior of $S_1 \triangle L$, as desired.
 \end{enumerate}
\end{proof}

\end{document}